\documentclass[12pt]{amsart}
\usepackage{amscd,amsmath,amsthm,amssymb}
\usepackage{pstcol,pst-plot,pst-3d}
\usepackage{color}
\usepackage{pstricks}
\usepackage{stmaryrd}
\newpsstyle{fatline}{linewidth=1.5pt}
\newpsstyle{fyp}{fillstyle=solid,fillcolor=verylight}
\definecolor{verylight}{gray}{0.97}
\definecolor{light}{gray}{0.9}
\definecolor{medium}{gray}{0.85}
\definecolor{dark}{gray}{0.6}

 \def\NZQ{\mathbb}               
 \def\NN{{\NZQ N}}
 \def\QQ{{\NZQ Q}}
 \def\ZZ{{\NZQ Z}}

 \def\frk{\mathfrak}               

 \def\mm{{\frk m}}

 \def\MR{{\mathcal R}}

 \def\G{{\mathcal G}}

 \def\ab{{\mathbf a}}
 \def\bb{{\mathbf b}}
 \def\xb{{\mathbf x}}

 \def\cb{{\mathbf c}}

 \def\opn#1#2{\def#1{\operatorname{#2}}} 
 \opn\chara{char} \opn\length{\ell} \opn\pd{pd} \opn\rk{rk}
 \opn\projdim{proj\,dim} \opn\injdim{inj\,dim} \opn\rank{rank}
 \opn\depth{depth} \opn\grade{grade} \opn\height{height}
 \opn\embdim{emb\,dim} \opn\codim{codim}
 
 \opn\Tr{Tr} \opn\bigrank{big\,rank}
 \opn\superheight{superheight}\opn\lcm{lcm}
 \opn\trdeg{tr\,deg}
 \opn\reg{reg} \opn\lreg{lreg} \opn\ini{in} \opn\lpd{lpd}
 \opn\size{size} \opn\sdepth{sdepth}
 \opn\link{link}\opn\fdepth{fdepth}\opn\lex{lex}
 \opn\tr{tr}
 \opn\type{type}
 \opn\div{div} \opn\Div{Div} \opn\cl{cl} \opn\Cl{Cl}
 \opn\Spec{Spec} \opn\Supp{Supp} \opn\supp{supp} \opn\Sing{Sing}
 \opn\Ass{Ass} \opn\Min{Min}\opn\Mon{Mon}
 \opn\Ann{Ann} \opn\Rad{Rad} \opn\Soc{Soc}
 \opn\Im{Im} \opn\Ker{Ker} \opn\Coker{Coker} \opn\Am{Am}
 \opn\Hom{Hom} \opn\Tor{Tor} \opn\Ext{Ext} \opn\End{End}
 \opn\Aut{Aut} \opn\id{id}
 
 \opn\nat{nat}
 \opn\pff{pf}
 \opn\Pf{Pf} \opn\GL{GL} \opn\SL{SL} \opn\mod{mod} \opn\ord{ord}
 \opn\Gin{Gin} \opn\Hilb{Hilb}\opn\sort{sort}
 \opn\PF{PF}\opn\Ap{Ap}
 \opn\aff{aff} \opn
 \con{conv} \opn\relint{relint} \opn\st{st}
 \opn\lk{lk} \opn\cn{cn} \opn\core{core} \opn\vol{vol}  \opn\inp{inp} \opn\nilpot{nilpot}
 \opn\link{link} \opn\star{star}\opn\lex{lex}\opn\set{set}
 \opn\width{wd}
 \opn\Fr{F}
 \opn\QF{QF}
 \opn\G{G}
 \opn\type{type}\opn\res{res}
 \opn\gr{gr}
 
 \def\pot#1#2{#1[\kern-0.28ex[#2]\kern-0.28ex]}

 \opn\dirlim{\underrightarrow{\lim}}
 \opn\inivlim{\underleftarrow{\lim}}
 \let\union=\cup
 \let\sect=\cap

 \let\iso=\cong
 
 \let\Sect=\bigcap
 \let\Dirsum=\bigoplus
 
 \let\to=\rightarrow
 
 \def\Implies{\ifmmode\Longrightarrow \else
         \unskip${}\Longrightarrow{}$\ignorespaces\fi}
 \def\implies{\ifmmode\Rightarrow \else
         \unskip${}\Rightarrow{}$\ignorespaces\fi}
 \def\iff{\ifmmode\Longleftrightarrow \else
         \unskip${}\Longleftrightarrow{}$\ignorespaces\fi}

 \let\:=\colon
 \newtheorem{Theorem}{Theorem}[section]
 \newtheorem{Lemma}[Theorem]{Lemma}
 \newtheorem{Corollary}[Theorem]{Corollary}
 \newtheorem{Proposition}[Theorem]{Proposition}
 \newtheorem{Remark}[Theorem]{Remark}
 
 \newtheorem{Example}[Theorem]{Example}
 \newtheorem{Examples}[Theorem]{Examples}
 \newtheorem{Definition}[Theorem]{Definition}

 \let\epsilon\varepsilon
 \let\kappa=\varkappa
 \def\qed{\ifhmode\textqed\fi
       \ifmmode\ifinner\quad\qedsymbol\else\dispqed\fi\fi}
 \def\textqed{\unskip\nobreak\penalty50
        \hskip2em\hbox{}\nobreak\hfil\qedsymbol
        \parfillskip=0pt \finalhyphendemerits=0}
 \def\dispqed{\rlap{\qquad\qedsymbol}}
 
 \opn\dis{dis}
 \def\pnt{{\raise0.5mm\hbox{\large\bf.}}}
 
 \opn\Lex{Lex}

\begin{document}
\title {On the number of generators of powers of an ideal}

\author {J\"urgen Herzog,  Maryam Mohammadi Saem and Naser Zamani}

\address{J\"urgen Herzog, Fachbereich Mathematik, Universit\"at Duisburg-Essen, Campus Essen, 45117
Essen, Germany} \email{juergen.herzog@uni-essen.de}

\address{Maryam Mohammadi Saem,  Faculty of Science, University of Mohaghegh Ardabili, P.O. Box 179, Ardabil, Iran}
\email{m.mohammadisaem@yahoo.com}

\address{Naser Zamani, Faculty of Science, University of Mohaghegh Ardabili, P.O. Box 179, Ardabil, Iran}
\email{naserzaka@yahoo.com}

\dedicatory{ }

\begin{abstract}
We study the number of generators of ideals in regular rings and ask the question whether $\mu(I)<\mu(I^2)$ if $I$ is not a principal ideal, where  $\mu(J)$
denotes the number of generators of an ideal $J$. We provide lower bounds for the number of generators for the powers of an ideal and also show that the CM-type
of $I^2$ is $\geq 3$ if $I$ is a monomial ideal of height $n$ in $K[x_1,\ldots,x_n]$ and $n\geq 3$.
\end{abstract}

\thanks{}




\maketitle

\setcounter{tocdepth}{1}

\section*{Introduction}

Let $K$ be a field,  $S=K[x_1,\ldots,x_n]$  the polynomial ring in $n$ variables over $K$. We denote by $\mu(I)$ the number of generators of a graded ideal
$I\subset S$. It is well known (see  \cite{Ko}) that for a graded ideal $I$,   the function $f(k)= \mu(I^k)$, for $k\gg 0$,  is a polynomial in $k$  of degree
$\ell(I)-1$, where $\ell(I)$ denotes the analytic spread of $I$, i.e., the Krull dimension of the fiber ring $F(I)$ of $I$. This fact implies in particular, that
if $I$ is not a principal ideal, then $\mu(I^k)<\mu(I^{k+1})$ for  $k\gg 0$. But what is the behavior of $\mu(I^k)$  for small integers $k$? Is it true that
$\mu(I^k)<\mu(I^{k+1})$  for all $k$, when $I$ is not a principal ideal?
This question has a positive answer if all generators of $I$ are of the same degree, that is, if the ideal is equigenerated. The proof of this result is easy and
relies on the fact that in the given situation, the fiber ring of the ideal is a domain. This will be discussed in Section~\ref{one}.

The problem is much harder
if not all generators are of same degree,  and it can be phrased as well for an ideal in a regular local ring. It turns out that if $I$ is not equigenerated, then  $\mu(I)-\mu(I^2)$ may be as big as we want,  as the following families of ideals ,  communicated to us by Conca, show: take $I=(x^4,x^3y, xy^3, y^4)+(x^2y^2)J$ with $J=(z,t)^a$, for example. Then $I\subset K[x,y,z,t]$, $I^2=(x,y)^8$, and hence $\mu(I^2)=9$, while $\mu(I)=5+a$. In a preliminary version of the paper we conjectured that $n\geq 2$, and let $I\subset K[x_1,\ldots,x_n]$ be a  graded ideal of height $n$ or  an ideal in a regular local ring $R$ of dimension $\geq 2$ with $\height I=\dim R$. Then $\mu(I^k)<\mu(I^{k+1})$ for all integers $k\geq 0$. In the meantime,  together  with Eliahou,   we found for  any integer $m\geq 6$ an ideal  of height 2 in $K[x,y]$ such that $\mu(I)=m$ and $\mu(I^2)=9$.  The construction of such ideals will appear in a forthcoming paper.

One may even ask under which conditions $\mu(I\sect J)\leq \mu(IJ)$, which would then  in particular imply that $\mu(I)\leq \mu(I^2)$. It is shown in Corollary~\ref{yes} that for equigenerated monomial ideals of height $2$ in $K[x,y]$ one even has $\mu(I\sect J)< \mu(IJ)$.

In Section~\ref{one} we observe the simple fact that if $\height I\geq 2$ and $\depth F(I)>0$, then  $\mu(I^k)<\mu(I^{k+1})$  for all $k$. From this it  can be
deduced that if $I$ is equigenerated, then  $\mu(I^k)\leq \mu(I^{k+1})$ for all $k$, and that if $\mu(I^k)=\mu(I^{k+1})$  for some $k$, then $I$ is a principal
ideal, see Corollary~\ref{iff}.  Unfortunately, $\depth F(I)$ may be zero, if $I$ is not equigenerated. For example, the fiber ring of $I=(x^8,x^7y^2, x^6y^5,
x^2y^7,y^{10})$ has depth zero. On the other hand, Tony Puthenpurakal communicated to us that if $I\subset K[x,y]$ is an integrally closed graded ideal of height 2, then $F(I)$ is Cohen-Macaulay. In addition, Conca, De Negri  and Rossi give in \cite[Theorem 5.1]{CNR} a complete characterization  of contracted ideals whose associated graded ring and hence its fiber cone is Cohen-Macaulay. Thus these are  interesting cases when $\mu(I^k)<\mu(I^{k+1})$  for all $k$.

In Theorem~\ref{shalom}, it is shown that $\mu(IJ)\geq \mu(I)+\mu(J)-1$ for equigenerated graded ideals $I,J$ in the polynomial ring.  This yields   the
inequalities  $\mu(I^k)\geq k(\mu(I)-1)+1$ for all $k$, if $I$ is an equigenerated graded ideal in the polynomial ring. In particular, one has for such ideals
that $\mu(I^2)\geq 2\mu(I)-1$. This lower bound can be improved if one takes under consideration the number of variables of the polynomial ring. In the special
case that $I$  is an equigenerated  monomial ideal in $K[x_1,\ldots,x_n]$,  our Theorem~\ref{freiman} says that $\mu(I^2)\geq  \ell(I)\mu(I)-{\ell(I) \choose
2}$.   The proof of the theorem is based on a well-known theorem from additive number theory, due to Freiman \cite[Lemma~1.14]{Fr}.

Of course,  the  function $f(k)=\mu(I^k)$ is just the Hilbert function of the fiber ring  of $I$, which is known to be a rational function of the form
$Q(t)/(1-t)^{\ell(I)}$,  where $Q(t)=\sum_{i\geq 0}h_it^i$ is a polynomial. By  using these facts it can be easily seen that $\mu(I^2)\geq 2\mu(I)-1$,  if and
only if $h_2\geq 0$, and that  $\mu(I^k)\geq k(\mu(I)-1)+1$ for all $k\geq 1$, if $h_i\geq 0$ for all $i$, see Proposition~\ref{compare}. We also observe in
Proposition~\ref{difference} that for a proper graded ideal in the polynomial ring one always has that $\mu(I^2)=\ell(I)\mu(I)-{\ell(I)\choose 2}+h_2$. Thus a
comparison with Theorem~\ref{freiman} yields the interesting fact that $h_2\geq 0$ for any equigenerated monomial ideal, see Corollary~\ref{h2}.

Due to the fact that  $\mu(I^2)\geq 2\mu(I)-1$,  if $I$ is an equigenerated ideal, one may expect that this inequality is true   for any
ideal of maximal height with no restrictions on the degrees of the generators. However this is not the case. The following example was communicated to us by Eliahou: Let $I=(x^6,
x^5y^2, x^4y^3,x^2y^4,y^6)$. Then $\mu(I^2)=8$, while $2\mu(I)-1=9$.

It is of interest to find out when equality holds in the inequality given in Theorem~\ref{shalom}, that is, when  $\mu(IJ)= \mu(I)+\mu(J)-1$. In Section~\ref{three}, we give a complete answer to this
question, when $I$ and $J$ are monomial ideals in $K[x,y]$. It turns out that $\mu(IJ)=\mu(I)+\mu(J)-1$, if and only if $I=(x^a,y^a)^r$ and $J=(x^a,y^a)^s$ where
$a$, $r$ and $s$ are positive integers, see Theorem~\ref{new}. As a consequence of this theorem, we obtain in Corollary~\ref{true} the result that
$\mu(I^k)=k(m-1)+1$ for some $k\geq 2$ if and only if, $I=(x^a,y^a)^r$ for positive integers $a$ and $r$. It can also be seen that if $I$ is generated in degree
$d$, then $\mu(I^k)=k(m-1)+1$ for some $k\geq 2$, if and only if the reduction number of $I$ is one with respect to the ideal $J=(x^d,y^d)$, see
Corollary~\ref{truered}.

In Section~\ref{four} we study monomial ideals in $K[x,y]$ which are not necessarily equigenerated ideals. If $G(I)=\{x^{a_i}y^{b_i}\}_{i=1,\dots,m}$, one may
assume that $a_1>a_2>\ldots>a_m\geq 0$. Then $0\leq b_1<b_2<\ldots<b_m$. Conversely, any two sequences like this determine the minimal set of generators of a
monomial ideal of $K[x,y]$. Therefore, monomial ideals in $K[x,y]$ are in bijection to such sequences $\bf{a}$ and $\bf{b}$. In Proposition~\ref{convex} we show
that $\mu(I^2)=2\mu(I)-1$ if either the sequences $\bf{a}$ and $\bf{b}$ are both convex or both concave, and in Corollary~\ref{Crash} we show that any lexsegment
ideal $I$ in $K[x,y]$ satisfies $\mu(I^2)=2\mu(I)-1$. We also remark that corresponding statements do not hold for more than two variables.

For $G(I)=\{u_1,\ldots,u_m\}$ and any integers with $1\leq i\leq j\leq m$ we define the safe area of the monomial $u_iu_j$ as a set $S_{ij}$ of monomials
$u_ku_{\ell}\in I^2$ with $ k<i$ and $l\leq j$, or $k=i$ and $l<j$, or $k>i$ and $j\leq l$. The monomials $u_ku_{\ell}\in S_{ij}$ ($k\leq \ell$) have  the
property that $u_iu_j$ does not divide $u_ku_\ell$ and $u_ku_{\ell}$ does not divide $u_iu_j$. This concept turned out to be useful in the proof of
Proposition~\ref{convex} and Corollary~\ref{Crash}. It is also used to prove that $\mu(I^2)>\mu(I)$ for any non-principal monomial ideal in $K[x,y]$ with
$\mu(I)\leq 7$.

In the final section of this paper we discuss a special case of a question due to Huneke \cite{Hu}. An affirmative answer to  Huneke's question would imply the
following result: Let $I=JL$ be a product of two monomial ideals in $S=K[x_1,\ldots,x_n]$, and suppose that $\height I=n$, then $\type(S/I)\geq n$. Huneke's
theorem \cite[Theorem 2]{Hu}  implies that in this situation $S/I$ is not Gorenstein, in other words that $\type(S/I)\geq 2$. Here,  in Theorem~\ref{type},  we succeed to show that
$\type(S/I)\geq 3$.

\section{On the growth of the number of generators of a graded ideal in a polynomial ring}\label{one}

Let $I$ be a graded ideal in $S=K[x_1,\ldots,x_n]$. We denote by $\MR(I)=\Dirsum_{k\geq 0}I^kt^k\subset S[t]$ the Rees ring of $I$, and by $\mm=(x_1,\ldots,x_n)$
the graded maximal ideal of $S$. Then
$F(I)=\MR(I)/\mm \MR(I)= \Dirsum_{k\geq 0}(I^k/\mm I^k)t^k$ is called the {\it{fiber ring}} of $I$. Obviously, $F(I)$ is a standard graded $K$-algebra whose
Krull dimension is by definition the {\em analytic spread} $\ell(I)$ of $I$. The $k$th graded component of $F(I)$ is $I^k/\mm I^k$. Hence for the  Hilbert
function $\Hilb(k,F(I))$ of $F(I)$ we have $\Hilb(k,F(I))=\mu(I^k)$. This explains the behavior of the function $f(k)=\mu(I^k)$ for large $k$.

\begin{Proposition}
\label{strict}
Let $I\subset S$ be a  graded ideal.   If $\depth F(I)>0$, then $\mu(I^k)\leq \mu(I^{k+1})$ for all $k$, and if in addition $\height I\geq 2$, then $\mu(I^k)<
\mu(I^{k+1})$ for all $k$.
\end{Proposition}

\begin{proof}
Without loss of generality we may assume that $K$ is infinite. Then, since $\depth F(I)>0$,  there exists a linear form $x\in  F(I)$ which is a nonzerodivisor on
$F(I)$. Multiplication with $x$ induces an injective map $I^k/\mm I^k\to I^{k+1}/\mm I^{k+1}$ for all $k$. This implies that $\mu(I^k)\leq \mu(I^{k+1})$ for all
$k$. Suppose now that $\height I\geq 2$,  and suppose  $\mu(I^k)= \mu(I^{k+1})$ for some $k$. Then $I^{k+1}=xI^k$,  which implies that $\height I=1$, a
contradiction.
\end{proof}

\begin{Corollary}
\label{iff}
Let $I\subset S$ be an equigenerated graded ideal. Then
\begin{enumerate}
\item[(a)] $\mu(I^k)\leq \mu(I^{k+1})$ for all $k\geq 0$.
\item[(b)] The following conditions are equivalent:
\begin{enumerate}
\item[(i)] $I$ is a principal ideal;
\item[(ii)] $\mu(I^k)=\mu(I^{k+1})$ for all $k$;
\item[(iii)] $\mu(I^k)=\mu(I^{k+1})$ for some $k$.
\end{enumerate}
\end{enumerate}
\end{Corollary}

\begin{proof}
Since $I$ is equigenerated, say $I=(f_1,\ldots,f_m)$ with $\deg f_i=d$ for all $i$,  it follows that $F(I)\iso K[f_1,\ldots,f_m]\subset S$. In particular,
$\depth F(I)>0$. Thus (a) follows from Proposition~\ref{strict}.

(b) (i)\implies (ii) \implies (iii) is obvious.
(iii)\implies (i): Assume that $I$ is not a principal ideal. If  $\height I=1$, then $I=fJ$ with $J\neq S$ an ideal with  $\height J\geq 2$. Since
$\mu(I^k)=\mu(J^k)$ for all $k$, we may assume that $I$ itself has height $\geq 2$. But then, by Proposition~\ref{strict},  $\mu(I^k)< \mu(I^{k+1})$,
contradicting our assumption.
\end{proof}

The following corollary was communicated to us by Tony Puthenpurakal.

\begin{Corollary}
\label{tony}
Let $I\subset K[x,y]$ be a  graded integrally closed ideal of height $2$.  Then $\mu(I^k)<\mu(I^{k+1})$ for all $k\geq 0$.
\end{Corollary}

\begin{proof} By \cite[Lemma 9.1.1]{SH},  $IK\llbracket x,y\rrbracket$ is again integrally closed, and furthermore we have  $\mu(I^k)=\mu(I^kK\llbracket x,y\rrbracket)$ for all $k$. Thus we may
as well assume that $I$ is an integrally closed  height $2$ ideal in $K\llbracket x,y\rrbracket$. It follows that the associated graded ring $\gr_I(K\llbracket x,y\rrbracket)$ is Cohen--Macaulay,
see  \cite[Proposition 5.5]{LT}, \cite[Theorem 3.1]{HS}. By \cite[Theorem A, Proposition 2.6]{HM} this  implies that $I^2=JI$, where $J$ is a minimal reduction
of $I$. It follows  (\cite[Theorem 1]{Ki}) that $F(I)$ is Cohen--Macaulay. Hence the desired conclusion results from Proposition~\ref{strict}.

\end{proof}

Let $I$ and $J$ be graded ideals of height $\geq 2$. Is it true that $\mu(I\sect J)<\mu(IJ)$?  If, yes, then for  $I=J$ we  obtain  that $\mu(I)<\mu(I^2)$. We
will show in Section~\ref{three} that this question has a positive answer in some special cases. Here we  show

\begin{Proposition}\label{museum}
Let $I$ and $J$ be graded  ideals of height 2 in $K[x,y]$. Then $$\mu(I\sect J)\leq \mu(I)+\mu(J)-1.$$
\end{Proposition}

\begin{proof}
We first begin with a few general remarks. Suppose $L\subset S=K[x,y]$ is a graded ideal of height 2. Then its minimal graded free resolution is of the form
\[
0\to S^{m-1}\to S^m\to L\to 0.
\]
The rank of the last free module in the resolution is the Cohen-Macaulay type of $S/L$, see \cite[Proposition A.6.6.]{HH}. Thus we have that
$\mu(L)=\type(S/L)+1$. Furthermore $\type(S/L)$ is the smallest number $r$ such that $L=\sect_{i=1}^rQ_i$, where each $Q_i$ is an irreducible ideal, that is
where each $S/Q_i$ is Gorenstein, see \cite[Satz 1.33]{HK}.

Now suppose that $\mu(I)=r$ and $\mu(J)=s$. Then $I=\Sect_{i=1}^{r-1}Q_i$ with irreducible $Q_i$ and $J=\Sect_{j=1}^{s-1}Q_j'$ with irreducible ideals $Q_j'$. It
follows $I\sect J=\Sect_{i=1}^{r-1}Q_i\sect \Sect_{j=1}^{s-1}Q_j'$ which implies that
\[
\mu(I\sect J)=\type(I\sect J)+1\leq (r-1)+(s-1)+1=\mu(I)+\mu(J)-1.
\]
\end{proof}

We introduce some notation that will be used in the statements and the  proofs of the next theorems. We fix a monomial order $<$.

\begin{Definition}\label{prime}
{\em Let $L$ be a graded ideal with $L_d\neq 0$ and $L_i=0$ for $i<d$. We denote by $L'$ the ideal generated by the monomials of  degree $d$ in $\ini_<(L)$.}
\end{Definition}

\begin{Theorem}\label{shalom} Let $I,J\subset K[x_1,\ldots,x_n]$ be equigenerated graded ideals. Then
\[
\mu(IJ)\geq \mu(I)+\mu(J)-1.
\]
In particular,
$\mu(IJ)> \max\{\mu(I), \mu(J)\},$ if $I$ and $J$ are not principal.
\end{Theorem}

\begin{proof}
Let $I$ be generated in degree $d_1$ and $J$ generated in degree $d_2$. Then $IJ$ is generated in the degree $d_1+d_2$.
We fix a monomial order $<$.  It follows from Macaulay's theorem \cite[Corollary 4.2.4]{BH}, that $\mu(I')=\mu(I)$,   $\mu(J')=\mu(J)$ and $\mu((IJ)')=\mu(IJ)$.
Assume we can  prove the desired inequality for equigenerated monomial ideals. Then
\[
\mu(IJ)=\mu((IJ)')\geq \mu(I'J')\geq \mu(I')+\mu(J')-1=\mu(I)+\mu(J)-1.
\]
Here we used that $I'J'\subseteq (IJ)'$.

Thus it remains to prove the theorem when $I$ and $J$ are monomial ideals. Let $G(I)=\{u_1,\ldots,u_r\}$ and $G(J)=\{v_1,\ldots,v_s\}$. We may assume that $s\geq
r$,  and that $u_1<u_2<\ldots <u_r$ and $v_1<v_2<\ldots <v_s$. Then
\begin{eqnarray*}
&&u_1v_1<u_1v_2< u_2v_2< u_2v_3 <u_3v_3< \ldots <u_{r-1}v_{r-1}< u_{r-1}v_r <u_rv_r \\
&<&u_rv_{r+1} <u_rv_{r+2} <\ldots <u_rv_{s}.
\end{eqnarray*}
These elements are pairwise distinct and of degree $d_1+d_2$, and hence part of a minimal set of generators of $IJ$. Counting this set of elements we find that
they are $(2r-1)+(s-r)=r+s-1$ elements. This yields the desired conclusion. \end{proof}

For an interesting class of ideals, which are not necessarily equigenerated, equality holds in Theorem~\ref{shalom}. Indeed, let $R$ be a two dimensional polynomial ring over an infinite  field $K$ or a two dimensional regular local ring with infinite residue class field, and let $I\subset R$ be a proper ideal which we assume to be graded if $R$ is the polynomial ring. According to Zariski \cite{Za}, the ideal $I$ is   said to be contracted,  if there exists $x\in \mm\setminus  \mm^2$  such that $I = IS \sect R$, where $S = R[\mm/x]$.

We denote by $o(I)$ the order of $I$, that is, the  largest integer $k$ such that $I\subseteq \mm^k$. If $I$ is graded, then $o(I)$ is just the least degree of a generator of $I$.

\medskip
Contracted ideals have the following nice properties:
\begin{itemize}
\item[(i)] $\mu(I)=o(I)+1$;
\item[(ii)] if $I$ and $J$ are contracted, then $IJ$  is contracted.
\end{itemize}

The proof of (i) can be found in \cite[Proposition 2.3]{Hu1}, where it is shown that (i) actually  characterizes contracted ideals,  and in the same article the proof of  (ii) can be found, see \cite[Proposition 2.6]{Hu1}.

\medskip
By using (i) and (ii) and observing that $o(IJ)=o(I)+o(J)$  we immediately obtain

\begin{Proposition}
\label{huneke}
Let $R$ a polynomial ring over an infinite  field $K$ or two dimensional regular local with infinite residue class field, and let $I, J\subset R$ be  ideals of height $2$ which we assume to be graded if $R$ is the polynomial ring. Then $\mu(IJ)=\mu(I)+\mu(J)-1$.
\end{Proposition}

\begin{Corollary}\label{kill}
Let $I_1,\ldots,I_r$ be equigenerated graded ideals. Then
$$\mu(I_1\cdots I_r)\geq\sum_{j=1}^r\mu(I_j)-(r-1).$$
In particular, $\mu(I^k)\geq k(\mu(I)-1)+1$, for all $k$.

Moreover, if all the ideal $I_j$ are contracted ideals,  then equality holds for both inequalities.
\end{Corollary}

\begin{proof}
We will prove the inequality by induction on $r$. The case $r=1$ is trivial. Now assume that the assertion holds for $r=k$, we will prove it for $r=k+1$. Clearly
$I_1\cdots I_{k+1}=(I_1\cdots I_k)I_{k+1}$. Then, by Theorem~\ref{shalom} and our induction hypothesis we have
\begin{eqnarray*}
\mu(I_1\cdots I_{k+1})&\geq &\mu(I_1\cdots I_k)+\mu(I_{k+1})-1\geq \sum_{i=1}^k\mu(I_i)-(k-1)+\mu(I_{k+1})-1\\
&=&\sum_{i=1}^{k+1}\mu(I_i)-((k+1)-1),
\end{eqnarray*}
as desired.

The statement for contracted ideals is an obvious consequence of Proposition~\ref{huneke}.
\end{proof}

Notice that for a contracted ideal,  Corollary~\ref{kill} yields  $$\mu(I^{k+1})-\mu(I^k)=\mu(I)-1>0 \quad \text{for all} \quad k.$$
Thus, since any contracted ideal is integrally closed, we have a new proof of Corollary~\ref{tony}.

\medskip
Corollary~\ref{kill} implies in particular that $\mu(I^2)\geq 2\mu(I)-1$. This lower bound is never reached when the analytic spread  of the ideal $I'$ (see
Definition~\ref{prime}) is $\geq 3$ and $I$ has a high number of generators. Thus for these situations we need a better lower bound. For this purpose we  recall a theorem in additive number theory,  due  to Freiman
\cite{Fr}: Let $S$ be a finite subset of $\ZZ^n$, and let $A(S)$ be the affine hull of the set $S$, that is, the smallest affine subspace of $\QQ^n$ containing
$S$.  The  dimension $d$ of $A(S)$ is called the {\em Freiman dimension} of $S$. The doubling of $S$ is the set $2S=\{a+b\: a,b\in S\}$. Now the theorem of
Freiman states that
\[
|2S|\geq (d+1)|S|-{d+1 \choose 2}.
\]
We use Freiman's theorem to prove

\begin{Theorem}
\label{freiman}
Let $I\subset K[x_1,\ldots,x_n]$  be an equigenerated graded  ideal. Then
\[
\mu(I^2)\geq \ell(I')\mu(I)-{\ell(I') \choose 2}.
\]
In particular, if $I$ is a monomial ideal, then $\ell(I')$ can be replaced by $\ell(I)$ in the above inequality.
\end{Theorem}

\begin{proof}
Since $\mu(I)=\mu(I')$ and $\mu(I^2)=\mu(I'^2)$ we may as well assume that $I$ is a monomial ideal. For a monomial ideal $J$ with unique minimal set of monomial
generators $G(J)$, we denote by $S(J)\subset \ZZ^n$ the set of exponent vectors of the elements of $G(J)$. Since $I$ and $I^2$ are  equigenerated, it follows
that $\mu(I)=|S(I)|$ and $\mu(I^2)=|S(I^2)|=|2S(I)|$. Let $d$ be the Freiman dimension of $S(I)$. The theorem will follow from Freiman's inequality once we have
shown that $\ell(I)=d+1$.

Let $S(I)=\{\cb_1,\ldots,\cb_m\}$ and $C$ the matrix whose row vectors are $\cb_1,\ldots,\cb_m$. Then $F(I)=K[\xb^{\cb_1},\ldots,\xb^{\cb_m}]$, and hence
$\ell(I)=\dim F(I)=\rank C$, see \cite[Lemma~10.3.19]{HH}. On the other hand,  the affine space spanned by $S(I)$ is of dimension $\rank C-1$. This concludes the
proof of the theorem.
\end{proof}

\section{The $h$-vector of the fiber ring}\label{two}

Let $(R,\mm,K)$ be a Noetherian  local ring or a standard graded $K$-algebra, $I\subset R$ a proper ideal which we  assume to be graded if $R$ is standard
graded. Let $d$ be the analytic spread of $I$. Then $F(I)$ is of Krull dimension $d$ and
\[
\Hilb_{F(I)}(t)=Q(t)/(1-t)^d,
\]
where
\[
Q(t)=\sum_{i\geq 0}h_it^i
\]
is a polynomial with $h_0=1$. One calls $(h_0,h_1,\ldots)$ the {\em $h$-vector} of $F(I)$.

\medskip
We want to interpret  Theorem~\ref{freiman} in relation to the $h$-vector of $F(I)$.

\begin{Proposition}
\label{compare}
Let $I\subset R$ be an ideal with  $\ell(I)=2$. Then
\begin{enumerate}
\item[(a)] $\mu(I^2)\geq 2\mu(I)-1$,  if and only if $h_2\geq 0$.
\item[(b)] $\mu(I^k)\geq k(\mu(I)-1)+1$ for all $k\geq 1$, if $h_i\geq 0$ for all $i$.
\item[(c)] $\mu(I^k)=  k(\mu(I)-1)+1$ for all $k\geq 1$, if and only if $h_i=0$ for all $i\geq 2$.
\end{enumerate}
\end{Proposition}

\begin{proof} (a) we have
\begin{eqnarray}\label{htwo}
\Hilb_{F(I)}(t)&=&\sum_{k\geq 0}\mu(I^k)t^k=Q(t)/(1-t)^2\\
&=&(1+h_1t+h_2t^2+\cdots )(1+2t+3t^2+\cdots).\nonumber
\end{eqnarray}
It follows that $\mu(I)=2+h_1$ and $\mu(I^2)=3+2h_1+h_2$. Therefore,
$$\mu(I^2)-(2\mu(I)-1)=h_2.$$ This yields the desired conclusion.

(b) By (\ref{htwo}) we have $\mu(I^k)=(k+1)+kh_1+ \sum_{i\geq 1}(k-i)h_{i+1}$. Hence
\[
\mu(I^k)-(k(\mu(I)-1)+1) =(k+1)+kh_1+ \sum_{i\geq 1}(k-i)h_{i+1}-k(h_1+1)-1=  \sum_{i\geq 1}(k-i)h_{i+1}.
\]
Since $ \sum_{i\geq 1}(k-i)h_{i+1}\geq 0$, if all $h_i\geq 0$, the assertion follows.

(c) It follows from the proof of (b) that $\mu(I^k)=  k(\mu(I)-1)+1$ for all $k\geq 1$, if  $h_i=0$ for all $i\geq 2$.

Conversely, suppose that  $\mu(I^k)=  k(\mu(I)-1)+1$ for all $k\geq 1$. Then
\[
\Hilb_{F(I)}(t)=\sum_{k\geq 0}(k(\mu(I)-1)+1)t^k=(1+(\mu(I)-2)t)/(1-t)^2.
\]
This shows that $h_i=0$  for $i\geq 2$.
\end{proof}

\begin{Corollary}
\label{hpositive}
Let $I\subset K[x,y]$  be an equigenerated graded ideal. Then $h_2\geq 0$.
\end{Corollary}

\begin{proof}
We may assume that $I$ is not a principal ideal, because otherwise the statement is trivial. If $\height(I)=1$, then $I=fJ$ where $J$ is an equigenerated ideal
of height $2$.  By Corollary~\ref{kill} we have $\mu(I^2)\geq 2\mu(I)-1$. Since  $\ell(I)=\ell(J)=2$, it follows from Proposition~\ref{compare} that $h_2\geq
0$.
\end{proof}

\begin{Example}\label{failure}
{\em If $I$ is not equigenerated, then the inequality $\mu(I^2)\geq 2\mu(I)-1$ may fail.
The following example was communicated to us by Eliahou.

\medskip
Let $I=(x^6, x^5y^2, x^4y^3,x^2y^4,y^6)$. Then $\mu(I^2)=8$, while $2\mu(I)-1=9$.

With CoCoA it can be checked that
\begin{eqnarray*}
 F(I)\iso K[z_1,\ldots,z_5]/(z_1z_5, z_2z_5, z_4^2 - z_3z_5, z_1z_4, z_2z_4, z_1z_3).
\end{eqnarray*}

The series $1+\sum_{k\geq 1}(3k+2)t^k$ is the Hilbert series  of $F(I)$. It follows that
\[
\mu(I^k)=k(\mu(I)-2)+2< k(\mu(I)-1)+1\quad  \text{for all $k\geq 2$}.
\]
One can also check that $\dim F(I)=2$ and $\depth F(I)=1$, and so $F(I)$ is not Cohen--Macaulay. The situation is much better if $F(I)$ is Cohen--Macaulay, as
the next result shows. }
\end{Example}

\begin{Corollary}
\label{cm}
Let $I\subset R$ be an ideal with  $\ell(I)=2$,  and assume that $F(I)$ is Cohen--Macaulay. Then  $\mu(I^k)\geq k(\mu(I)-1)+1$ for all $k\geq 1$.
\end{Corollary}

\begin{proof}
If $F(I)$ is Cohen--Macaulay, then $h_i\geq 0$ for all $i$, see \cite[Corollary 4.1.10]{BH}. Thus the desired conclusion follows from Proposition~\ref{compare}.
\end{proof}

The next proposition gives an interpretation of the difference between the right side and left side term in the inequality of Theorem~\ref{freiman}.

\begin{Proposition}
\label{difference}
Let $I\subset R$ be a proper  ideal. Then
\[
\mu(I^2)=\ell(I)\mu(I)-{\ell(I)\choose 2}+h_2.
\]

\end{Proposition}

\begin{proof}
For the fiber ring of the ideal $I$ we have
\begin{eqnarray*}
\Hilb_{F(I)}(t)&=&\frac{(1+h_1t+h_2t^2+\ldots)}{(1-t)^{\ell(I)}}\\
&=&{(1+h_1t+h_2t^2+\ldots)}(1+\ell(I)t+{\ell(I)+1\choose 2}t^2+\ldots)\\
&=& 1+(h_1+\ell(I))t+(h_1\ell(I)+{\ell(I)+1\choose 2}+h_2)t^2+\ldots .
\end{eqnarray*}
It follows that $\mu(I)=h_1+\ell(I)$ and $\mu(I^2)=h_1\ell(I)+{\ell(I)+1\choose 2}+h_2$, which yields the desired result.
\end{proof}

Comparing Theorem~\ref{freiman} with Proposition~\ref{difference}, we obtain
\begin{Corollary}\label{h2}
Let $I\subset K[x_1,\ldots ,x_n]$ be an equigenerated monomial ideal. Then $h_2\geq 0$.
\end{Corollary}

\section{Equigenrated Monomial ideals in $K[x,y]$}\label{three}

In this section we consider in more details the case that $I$ is a monomial ideal in two variables, and classifies those ideals among them which have the
property that equality holds in the inequalities given in Theorem~\ref{shalom} and Corollary~\ref{kill}.

Let $K$ be a field and $I\subset K[x,y]$ be a monomial ideal. The unique minimal set of monomial generators will be denoted by $G(I)$. Let
$G(I)=\{x^{a_i}y^{b_i}\}_{i=1,\ldots,m}$. We may assume $a_1>a_2>\ldots>a_m\geq 0 $. Then $0\leq b_1<b_2<\ldots<b_m$. Conversely, given sequences of integers
\begin{eqnarray}\label{sequences}
a_1>a_2>\ldots>a_m\geq 0 \quad \text{and} \quad 0\leq b_1<b_2<\ldots<b_m,
\end{eqnarray}
they define the monomial ideal $I$ with $G(I)=\{x^{a_i}y^{b_i}\}_{i=1,\ldots,m}$. Therefore, monomial ideals in $K[x,y]$ are in  bijection to  sequences as in
(\ref{sequences}).

We will always assume that $m>1$. We are interested in the number of generators of $I^k$.  Since  $I=x^{a_m}y^{b_1}J$, where $J$  is a monomial  ideal of height
$2$ containing the pure powers $x^{a_1-a_m}$ and $y^{b_m-b_1}$, and  since $\mu(I^k)=\mu(J^k)$ for all $k$,  for most statements to follow, it is not an
restriction to assume that $\height(I)=2$.

If  $I$ is an equigenerated ideal generated in degree $d$,  then
\[
G(I)=\{x^{a_i}y^{d-a_i}\}_{i=1,\ldots,m}\quad \text{with}\quad d=a_1>a_2>\ldots> a_m=0.
\]
This case is significantly simpler than the general case, since only one sequence of integers determines the ideal.

More generally, let $A=\{a_1,a_2,\ldots,a_m\}$ be a set of non-negative integers,  $d$ a positive integer with $a_i\leq d$ for all $i$, and let
$I=(x^{a_1}y^{d-a_1},\ldots, x^{a_m}y^{d-a_m})$. Then $G(I)=\{x^{a_1}y^{d-a_1},\ldots, x^{a_m}y^{d-a_m}\}$ and $\mu(I)=m$.

Let $I$ and $J$ be equigenerated monomial ideals in $S=K[x,y]$. We know from Theorem~\ref{shalom} that $\mu(IJ)\geq \mu(I) + \mu(J) -1$. The following result
tells us when we have equality.

\begin{Theorem}\label{new}
Let $I,J\subset K[x,y]$ be equigenerated monomial ideals. The following conditions are equivalent:
\begin{enumerate}
\item [(i)] $\mu(IJ)=\mu(I)+\mu(J)-1$;
\item [(ii)] There exist positive integers $a$, $r$ and $s$ such that $I=(x^a,y^a)^r$ and $J=(x^a,y^a)^s$.
\end{enumerate}
\end{Theorem}

\begin{proof}
Let us assume that $I$ is generated in degree $d_1$ and $J$ is generated in degree $d_2$. Observe that
\begin{eqnarray}
\label{a}G(IJ)\supseteq x^{d_2}G(I)\cup G(J)y^{d_1}\\
x^{d_2}G(I)\cap G(J)y^{d_1}=\{x^{d_2}y^{d_1}\}.\label{b}
\end{eqnarray}

Proof of (\ref{a}): We first notice that the elements of $x^{d_2}G(I)$ and the elements of $G(J)y^{d_1}$ are monomials of degree $d_1+d_2$ in $IJ$. Since the
monomials of $x^{d_2}G(I)$ are pairwise different it follows  that $x^{d_2}G(I)\subseteq G(IJ)$. The same argument holds for $G(J)y^{d_1}$.

Proof of (\ref{b}): All elements of $x^{d_2}G(I)$ are of the form $x^{d_2+a}y^{d_1-a}$ and all elements of $G(J)y^{d_1}$ are of the form $x^by^{d_2-b+d_1}$ with
$b\leq d_2$. It follows $x^{d_2+a}y^{d_1-a}=x^by^{d_2-b+d_1}$ if and only if $b=d_2+a$. Since $b\leq d_2$, this is only possible if $b=d_2$ and $a=0$.

\medskip
Now we prove the equivalence of (i) and (ii).

(i)\implies (ii): Let $G(I)=\{x^{a_i}y^{d_1-a_i}\}_{i=0,\ldots,r}$ with $a_0=0<a_1<\ldots<a_r=d_1$ and  $G(J)=\{x^{b_j}y^{d_2-b_j}\}_{j=0,\ldots,s}$ with
$b_0=0<b_1<\ldots<b_s=d_2$. Since $\mu(IJ)=\mu(I)+\mu(J)-1$, it follows from (\ref{a}) that
$$G(IJ)=\{x^{a_i+d_2}y^{d_1-a_i}\}_{i=0,\ldots,r}\cup\{x^{b_j}y^{d_1+d_2-b_j}\}_{j=0,\ldots,s},$$
and we have
\begin{eqnarray}\label{sequence}
b_0<b_1<\ldots<b_s=d_2<d_2+a_1<\ldots<d_2+a_r=d_1+d_2.
\end{eqnarray}
Note that for $i=0,\ldots,s-1$,
$$x^{b_i+a_1}y^{d_1+d_2-(b_i+a_1)}=(x^{b_i}y^{a_2-b_i})(x^{a_1}y^{d_1-a_1})$$
belongs to $G(IJ)$.

Since $0<b_0+a_1<b_2+a_1<\ldots<b_{s-1}+a_1<d_2+a_1$, in comparison with (\ref{sequence}) shows that $b_j=b_{j-1}+a_1$ for $j=0,\ldots,s-1$. This implies that
$b_j=ja_1$ for $j=0,\ldots,s$. In particular, we have $b_1=a_1$. By symmetry we also have $a_i=ib_1$ for $i=0,\ldots,r$. Hence if we set $a=a_1=b_1$, we finally
get $b_j=ja$ for $j=1,\ldots,s$ and $a_i=ia$ for $i=0,\ldots,r$.

(ii) \implies (i): By assumption $IJ=(x^a,y^a)^{r+s}$. Therefore,
$$\mu(IJ)=r+s+1=(r+1)+(s+1)-1=\mu(I)+\mu(J)-1.$$
\end{proof}

\begin{Corollary}\label{yes}
Let $I$ and $J$ be equigenerated monomial ideals of height $2$ in $K[x,y]$.  Then  $\mu(I\sect J)<\mu(IJ)$.
\end{Corollary}

\begin{proof}
By Proposition~\ref{museum} and Theorem~\ref{new} we have
\[
\mu(I\sect J)\leq \mu(I)+\mu(J)-1\leq \mu(IJ).
\]
Suppose $\mu(I\sect J)= \mu(IJ)$. Then  there exist positive integers $a$, $r$ and $s$. such that $I=(x^a,y^a)^r$ and $J=(x^a,y^a)^s$. We may assume that $r\leq
s$. Then $I\sect J=(x^a,y^a)^s$ and $IJ=(x^a,y^a)^{r+s}$.  It follows that $\mu(I\sect J)=s+1$ and $\mu(IJ)=r+s+1$, a contradiction.
\end{proof}

All assumptions given in Corollary~\ref{yes} are required to guarantee that the inequality $\mu(I\sect J)<\mu(IJ)$ is indeed strict. The strict inequality fails,
if the ideals $I$ and $J$ are not equigenerated. For example, if $I=(x^2,y)$ and $J=(x,y^2)$, then $I\cap J=(x^2,xy,y^2)$ and $IJ=(x^3,xy,y^3)$. Therefore
$\mu(I\cap J)=\mu(IJ)$. On the other hand, if both of the ideals are equigenerated, but not of height $2$, the strict inequality will fail again. For example,
let $I=(x^3,xy^2)$ and $J=(x^2y,y^3)$, then $I\cap J=(x^3y,x^2y^2,xy^3)$ and $IJ=(xy^5,x^3y^3,x^5y)$. Hence $\mu(I\cap J)=\mu(IJ)$.

We do not know whether $\mu(I\cap J)\leq \mu(IJ)$ for any two graded  ideals in $K[x,y]$,  even if both of them are equigenerated.



\begin{Corollary}\label{induction}
Let $I_1,\ldots,I_r$ be equigenerated monomial ideals in $S=K[x,y]$ with $\height I_j=2$ for $j=1,\ldots,r$.  Then
$$\mu(I_1\ldots I_r)\geq \sum_{j=1}^r\mu(I_j)-(r-1).$$
Equality holds if and only if either $r=1$ or $r>1$ and there exist positive integers $a$ and $s_j$ such that $I_j=(x^a,y^a)^{s_j}$ for $j=1,\ldots,r$.
\end{Corollary}

\begin{proof}
We will prove the inequality by induction on $r$. The case $r=1$ is trivial. Now assume that the assertion holds for $r=k$, we will prove it for $r=k+1$. Clearly
$I_1\ldots I_{k+1}=(I_1\ldots I_k)I_{k+1}$. Then, by Theorem~\ref{shalom} and our induction hypothesis we have,
\begin{eqnarray*}
\mu(I_1\ldots I_{k+1})&\geq &\mu(I_1\ldots I_k)+\mu(I_{k+1})-1\geq \sum_{i=1}^k\mu(I_i)-(k-1)+\mu(I_{k+1})-1\\
&=&\sum_{i=1}^{k+1}\mu(I_i)-((k+1)-1),
\end{eqnarray*}
as desired.

Now we prove the second part of the corollary. For $r=1$ the assertion is trivial. For $r=2$ the statement follows from Theorem~\ref{new}(b). Now assume that the
statement holds for $r=k$ with $k\geq 2$.
By our assumption and Theorem~\ref{new}(a), we have
$$\sum_{j=1}^{k+1}\mu(I_j)-k=\mu((I_1\ldots I_k)I_{k+1})\geq\mu(I_1\ldots I_k)+\mu(I_{k+1})-1.$$
It follows that $\mu(I_1,\ldots,I_k)\leq \sum_{j=1}^k\mu(I_j)-(k-1)$. Since the opposite inequality  always holds, we get
$$\sum_{j=1}^{k}\mu(I_j)-(k-1)=\mu(I_1\ldots I_k).$$
 So by induction hypothesis there exist positive integers $a$, $s_j$ for $j=1,\ldots,k$ such that $I_j=(x^a,y^a)^{s_j}$. Similarly if we write $I_1\ldots
 I_{k+1}=I_1(I_2\ldots I_{k+1})$ then, by the same argument as above we see that there exists $s_{k+1}$ such that $I_{k+1}=(x^a,y^a)^{s_{k+1}}$.
\end{proof}

As an immediate consequence of Corollary~\ref{induction}, we have
\begin{Corollary}
\label{true}
Let $I\subset S=K[x,y]$ be an equigenerated monomial ideal with $\mu(I)=m$ and $\height I=2$. Then
\begin{enumerate}
\item[(a)] $\mu(I^k)\geq k(m-1)+1$;
\item[(b)] The following conditions are equivalent:
\begin{enumerate}
\item[(i)] $\mu (I^k)=k(m-1)+1$ for some $k\geq 2$.
\item[(ii)] $\mu (I^k)=k(m-1)+1$ for all $k\geq 2$.
\item[(iii)] There exist positive integers $a$ and $r$ such that $I=(x^a,y^a)^r$.
\end{enumerate}
\end{enumerate}
\end{Corollary}

A result corresponding to (b)(i)\iff (ii) is no longer valid for more than two variables. In other words, if $I\subset K[x_1,\ldots,x_n]$ is a monomial ideal of
height $n$, then $\mu(I^2)=n\mu(I)-{n \choose 2}$ does not imply that $I=(x_1^a,\ldots,x_n^a)^r$ for some positive integers $a$ and $r$. For example, if
$I=(x,y,z)^3$, then $\mu(I)=10$ and  $\mu(I^2)=28>3\cdot10-{3\choose 2}=27$.

\medskip
The following corollary gives another characterization for ideal whose powers achieve the lower bound for the number of generators.

\begin{Corollary}
\label{truered}
Let $I\subset S=K[x,y]$ be a monomial ideal generated in  degree  $d$  with $\mu(I)=m$ and $\height I=2$. Let $J=(x^d,y^d)$.  Then the following conditions are
equivalent:
\begin{enumerate}
\item[(i)] $\mu (I^k)=k(m-1)+1$ for some $k\geq 2$.
\item[(ii)] $I^2=JI$.
\end{enumerate}
\end{Corollary}

\begin{proof}
(i)\implies (ii): By Corollary~\ref{true}, our hypothesis implies that $\mu (I^2)=2(m-1)+1$. By (\ref{a}) we have $G(I^2)\supseteq x^{d}G(I)\cup G(J)y^{d}$, and
therefore (\ref{b}) implies that $| x^{d}G(I)\cup G(J)y^{d}|=2(m-1)+1$. It follows that $I^2=x^dI+Iy^d=JI$.

(ii)\implies (i): The assumption implies that $I^2=x^dI+Iy^d$, and hence $G(I^2)\subseteq x^{d}G(I)\cup G(I)y^{d}$. By  (\ref{a}),  $G(I^2)\supseteq
x^{d}G(I)\cup G(I)y^{d}$, and hence equality holds. Thus together with  (\ref{b}) it follows that $\mu(I^2)=2(m-1)+1$.
\end{proof}

Corollary~\ref{truered} says that $\mu (I^k)=k(m-1)+1$ for some $k\geq 2$ if and only if the reduction number of $I$ is equal to $1$.

\medskip
We conclude this section with a result which says that the powers of equigenerated ideals of same degree can be exchanged with their sums.

\begin{Proposition}\label{cold}
Let $I_1,\ldots,I_r$ be equigenerated monomial ideals in $S=K[x,y]$, all generated in  degree $d$ such that $\height I_{j}=2$ for all $j=1,\ldots ,r$ and
$I_1+\ldots+I_r=(x,y)^d$. Then $(I_1+\ldots+I_r)^k=I_1^k+\ldots+I_r^k$ for all $k\geq 1$.
\end{Proposition}

\begin{proof}
We will prove the statement by induction on $k$. For $k=1$ the assertion is trivial. Let $k\geq 2$. Let $x^ay^{kd-a}\in (x,y)^{kd}$, and let us first assume
that $a\geq d$. We have $x^ay^{kd-a}=x^d(x^{a-d}y^{kd-a})$. By induction hypothesis $I_1^{k-1}+\ldots+I_r^{k-1}=(x,y)^{(k-1)d}$. Therefore
$x^{a-d}y^{kd-a}\in I_j^{k-1}$ for some $j=1,\ldots,r$. Moreover, $x^d\in I_j$. This implies that $x^ay^{kd-a}\in I_j^{k}$. On the other hand, if $a<d$,
then $kd-a\geq d$. Therefore, $x^ay^{kd-a}=y^d(x^ay^{(k-1)d-a})$. By similar argument as before we see that $x^ay^{kd-a}\in I_j^{k}$, for some $j$.
\end{proof}

\begin{Examples}
{\em (a) Proposition~\ref{cold} is no longer valid if $\height I_{j}\neq 2$ for some $j$. For example, let $I_1=(x^2,y^2)$ and $I_2=(xy)$. Then
$I_1+I_2=(x,y)^2$, but $I_1^2+I_2^2\neq(x,y)^4$.

(b) Proposition~\ref{cold} is no longer valid in a polynomial ring with more than two variables, as the following example shows.
Let $$I=(x^5,y^5,z^5,xyz^3,xy^2z^2,x^2y^3,x^2z^3,x^3y^2,x^3yz,x^3z^2)$$ and $$J=(x^5,y^5,z^5,x^2y^2z,x^2yz^2,xz^4,xy^3z,yz^4,y^2z^3,y^3z^2,y^4z, xy^4, x^4z,
x^4y).$$

Then $I+J=(x,y,z)^5$, but $x^3y^3z^4\notin I^2+J^2$.
}
\end{Examples}

\section{Monomial ideals in $K[x,y]$ which are not necessarily equigenerated}\label{four}

In the previous sections we assumed that $I$ is equigenerated. For such ideals we had that $\mu(I^2)\geq 2\mu(I)-1$. But we have seen in Example~\ref{failure}
that this inequality is in general  no longer true if $I$ is not equigenerated. We now give conditions on the exponents on the generators of $I$ which guarantee
that $\mu(I^2)=2\mu(I)-1$.

\medskip
For the rest of this section we assume $I\subset S=K[x,y]$ is a monomial ideal with $G(I)=\{u_1,\ldots,u_m\}$, where $u_i=x^{a_i}y^{b_i}$ for $i=1,\ldots ,m$
with  $a_1>a_2>\ldots>a_m=0$ and $0=b_1<b_2<\ldots<b_m$.

The generators of $I^2$ can be displayed in the following triangle $T(I)$. Of course this set of  monomials displayed in this triangle is usually not minimal set
of generators.
\[
\begin{array}{cccccc}
u_1^2&u_1u_2&u_1u_3&\ldots& u_1u_{m-1}&u_1u_m\\
 &u_2^2&u_2u_3&\ldots&u_2u_{m-1}&u_2u_m\\
&& u_3^2&\ldots&u_3u_{m-1}&u_3u_m\\
&&& \ddots& \vdots&\vdots\\
&&&& u_{m-1}^2&u_{m-1}u_m\\
&&&&& u_{m}^2
\end{array}
\]
The multi-degrees of the generators in this diagram are displayed in the next diagram.
 \[
\begin{array}{cccccc}
(2a_1,0)&(a_1+a_2,b_2)&\ldots& (a_1+a_{m-1},b_{m-1})&(a_1,b_m)\\
 &(2a_2,2b_2)&\ldots&(a_2+a_{m-1},b_2+b_{m-1})&(a_2,b_2+b_m)\\
&& \ddots& \vdots&\vdots\\
&&& (2a_{m-1},2b_{m-1})&(a_{m-1},b_{m-1}+b_m)\\
&&&& (0,2b_m)
\end{array}
\]
For $k=2,\ldots,2m$ we call the set of monomials $$D_k=\{u_iu_j\:\;  1\leq i\leq j\leq m\text{ and } i+j=k\}$$   the $k$th diagonal of $T(I)$.
It is hard to predict where the minimal generators  of  $I^2$ are distributed in $T(I)$.  We will mark the  elements of $G(I^2)$ in $T(I)$  by bold letters.

\begin{Example}\label{diagonal}
{\em
(1) Let $G(I)=\{x^7,x^6y^2,x^5y^3,x^3y^4,y^7\}$. Then we have the following triangle for the generators of $I^2$:
 \[
\begin{array}{cccccc}
{\bf(14,0)}&{\bf(13,2)}&{\bf(12,3)}& {\bf(10,4)}&{\bf(7,7)}\\
 &(12,4)&(11,5)&{\bf(9,6)}&(6,9)\\
&& (10,6)&(8,7)&{\bf(5,10)}\\
&&& {\bf (6,8)}&{\bf(3,11)}\\
&&&& {\bf(0,14)}
\end{array}
\]
}
\end{Example}

As one can see, the elements in the diagonal $D_7$ whose multi-degrees are  $(8,7)$ and $(6,9)$ do not belong to $G(I^2)$.

\medskip
(2) Let $G(I)=\{x^7,x^6y^4,x^4y^5,x^3y^6,y^8\}$. Then $T(I)$ is
 \[
\begin{array}{cccccc}
{\bf(14,0)}&{\bf(13,4)}&{\bf(11,5)}& {\bf(10,6)}&{\bf(7,8)}\\
 &(12,8)&(10,9)&(9,10)&(6,12)\\
&& (8,10)&(7,11)&{\bf(4,13)}\\
&&& {\bf(6,12)}&{\bf(3,14)}\\
&&&& {\bf(0,16)}
\end{array}
\]
This example shows that in general it is not the case that in each row of $T(I)$  there is an element  of $G(I^2)$. The elements of the second row, whose
elements have multi-degree  $(12,8),(10,9), (9,10)$ and $(6,12)$ do not belong to $G(I^2)$.

For each $1\leq i\leq j\leq m$, we can consider an area $S_{ij}$ for the element  $u_iu_j$ in $T(I)$ which we call the  {\it safe area} for $u_iu_j$. It has the
property that for all  $u_ku_l\in S_{ij}$ one has: $u_ku_l$ does not divide $u_iu_j$,  and $u_iu_j$ does not divide  $u_ku_l$. It is easy to see that
\[
S_{ij}=\{u_ku_l\:\; \text{$k<i$ and $l\leq j$, or $k=i$ and $l<j$, or $k>i$ and $j\leq l$}\}.
\]

In Figure~\ref{safe}, the shadowed area shows the safe area of $u_iu_j$.

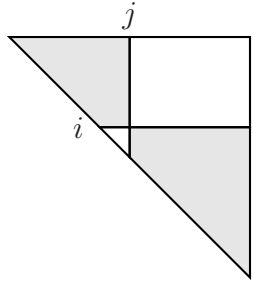
\begin{figure}[hbt]
\begin{center}
\psset{unit=0.4cm}
\begin{pspicture}(8,0)(3,9)
\pspolygon[style=fyp,fillcolor=light](0,8)(4,8)(4,5)(3,5)
\pspolygon[style=fyp,fillcolor=light](4,5)(8,5)(8,0)(4,4)
\pspolygon(4,8)(8,8)(8,5)(4,5)
\pspolygon(3,5)(4,5)(4,4)
\rput(2.3,5){$i$}
\rput(4,8.7){$j$}
\end{pspicture}
\caption{Safe Area}\label{safe}
\end{center}
\end{figure}

\medskip
We say that a sequence of integers $c_1,c_2,\ldots,c_m$ is {\em concave } resp.\ {\em convex}, if $2c_i\leq c_{i-1}+c_{i+1}$ resp.\ $2c_i\geq c_{i-1}+c_{i+1}$
for $i=2,\ldots,m-1$. The first  inequalities imply that
\begin{eqnarray}\label{Iraqidolmeh}
c_{i+k}+c_{j-k}\leq c_i+c_j \quad\text{for all $k$ with $0\leq k\leq (j-i)/2$.}
\end{eqnarray}
We obtain  similar inequalities in the convex  case.

\medskip
In contrast to Example~\ref{diagonal}(1), we have  $D_k\sect G(I^2)\neq \emptyset$  for $k=2,\ldots,2m$ for monomial ideals  $I$ as described in the next
result.

\begin{Proposition}
\label{convex}
Let ${\bf{a}}:a_1>a_2>\cdots >a_m=0$ and ${\bf{b}}:0=b_1<b_2<\cdots <b_m$ be integer sequences, and let $G(I)=\{x^{a_i}y^{b_i}\}_{i=1,\ldots,m}$. Suppose that
the sequences ${\bf{a}}$ and ${\bf{b}}$ either are both  convex or both  concave. Then each diagonal of $T(I)$ contains exactly one generator of $I^2$. In
particular, $\mu(I^2)=2\mu(I)-1$. More precisely,  we have $$G(I^2)=\{u_1^2,\ldots,u_m^2,u_1u_2,u_2u_3,\ldots,u_{m-1}u_m\},$$
if ${\bf{a}}$ and ${\bf{b}}$ are concave, and $$G(I^2)=\{u_1^2, u_1u_2,\ldots ,u_1u_m, u_2u_m,\ldots ,u_m^2\},$$ if ${\bf{a}}$ and ${\bf{b}}$ are convex.
\end{Proposition}

\begin{proof}
Let us assume that ${\bf{a}}$ and ${\bf{b}}$ are concave  sequences. The case that ${\bf{a}}$ and ${\bf{b}}$ are convex sequences is treated similarly.

Let $1\leq i\leq j\leq m$. In the first step we show that $u_iu_j$ is divisible by $u_{(i+j)/2}^2$ if $i+j$ is even, and $u_iu_j$ is divisible by
$u_{(i+j-1)/2}u_{(i+j+1)/2}$ if $i+j$ is odd.

Indeed, if $i+j$ is even, then $u_{(i+j)/2}^2=x^{2a_{(i+j)/2}}y^{2b_{(i+j)/2}}$. Since the sequences  ${\bf{a}}$ and ${\bf{b}}$ are concave, (\ref{Iraqidolmeh})
implies that $2a_{(i+j)/2}\leq a_i+a_j$ and $2b_{(i+j)/2}\leq b_i+b_j$. This implies that it follows that $u_{(i+j)/2}^2$ divides $u_iu_j$. If $i+j$ is odd, then
by using again the concavity condition, (\ref{Iraqidolmeh}) implies that $a_{(i+j-1)/2}+a_{(i+j+1)/2}\leq a_i+a_j$ and $b_{(i+j-1)/2}+b_{(i+j+1)/2}\leq b_i+b_j$.
It follows that $u_{(i+j-1)/2}u_{(i+j+1)/2}$ divides $u_iu_j$. Thus we have shown that
$$
G(I^2)\subseteq \{u_1^2,\ldots,u_m^2,u_1u_2,u_2u_3,\ldots,u_{m-1}u_m\}.
$$
Each monomial $v$ of the set ${\mathcal S}$ on the righthand side has a property that all the other monomials of ${\mathcal S}$ are in the safe area of $v$. This
shows that $G(I^2)={\mathcal S}$.
\end{proof}

Another class of ideals $I$ for which  the equation $\mu(I^2)=2\mu(I)-1$ holds, is the class of lexsegment  ideals.  Recall that an ideal $I\subset K[x_1,\ldots,x_n]$ is
called a {\em lexsegment ideal}, if for all monomial $u\in I$  and all monomials $v$ with $\deg v=\deg u$ and $v>u$ in  the lexicographical order, it follows
that $v\in I$.

The lexsegment ideals in $K[x,y]$ can be described as follows. Let $d\geq a\geq 1$ be integers, and let $I_{d,a}$ be the lexsegment ideal with
$$G(I_{d,a})=(x^{d-i}y^i)_{i=0,\ldots,a}.$$

\begin{Lemma}\label{Theresa May}
Let $I\subset K[x,y]$ be a monomial ideal. The following conditions are equivalent:
\begin{enumerate}
\item[(i)] $I$ is a lexsegment ideal;
\item[(ii)] There exist integers $0\leq a_1<a_2<\ldots < a_r$ such that $$I=I_{d,a_1}+I_{d+1,a_2}+\ldots+I_{d+r-1,a_r};$$
\item[(iii)] There exist  integers $0<b_1<\ldots<b_s$ such that $$I=(x^d,x^{d-1}y^{b_1},\ldots, x^{d-s}y^{b_s}).$$
\end{enumerate}
\end{Lemma}

\begin{proof}
(i)\implies (ii) follows from the fact that each $I_{d+i-1,a_{i}}$ is a lexsegment ideal, that $(x,y)I_{d+i-1,a_{i}}\subset I_{d+i,a_{i+1}}$ and that
\begin{eqnarray}\label{lex}
(x,y)^kI_{d,a}=I_{d+k,a+k} \quad \text{for all $k$}.
\end{eqnarray}

(ii)\implies (iii): It follows from (\ref{lex}) that $I=(x^a,x^{a-1}y^{b_1},\ldots, x^{a-s}y^{b_s})$ where the $b$-sequence is the set of integers
$$[0,a_1]\cup [a_1+2,a_2]\cup [a_2+2,a_3]\cup\ldots\cup [a_{r-1}+2,a_r],$$
in the natural order with suitable $a_i$. Here $[a,b]=\emptyset$ if $b<a$ and $[a,b]=\{i\in \NN \:\; a\leq i\leq b\}$.

(iii)\implies (i): It is enough to show that for all $u\in G(I)$ and all monomials $v$ with $deg(v)=deg(u)$ and $v>u$, we have $v\in I$. So let
$u=x^{d-i}y^{b_i}$ and for some $v=x^{d-(i-j)}y^{b_i-j}$ for some integer $j> 0$. Since $b_i-j\geq b_{i-j}$, we have
$$ v= x^{d-(i-j)}y^{b_i-j}=y^{(b_i-j)-b_{i-j}}(x^{d-(i-j)}y^{b_{i-j}})\in I,$$
as desired.
\end{proof}

In the following we consider a class of ideals which is more general than lexsegment ideals.
\begin{Theorem}\label{back}
Let $a$ be a positive integer and $I_1,\ldots,I_r\subset S=k[x,y]$ be monomial ideals with $G(I_j)=\{x^{ia}y^{b_{i,j}}\}_{j=s_j,\ldots,t_j}$, where $0\leq
s_j\leq t_j$. Then there exists integers $c_i$ such that

$$G(I_1\cdots I_r)=\{x^{ia}y^{c_i}\}_{i=s,\ldots,t},$$
where $s=s_1+\ldots+s_r$ and $t=t_1+\ldots+t_r$.
\end{Theorem}

\begin{proof}
First note that $b_{i,j}>b_{i+1,j}$ for all $i$ and $j$. It is clear that
$$c_i=\min\{b_{i_1,1}+\ldots+b_{i_r,r}\;\: i_1+\ldots+i_r=i\}.$$
 This set is a minimal set of generators of $I_1\cdots I_r$ if and only if $c_i>c_{i+1}$, for $i=s,\dots, t-1$. The case $s=t_1+\cdots +t_r$ is clear. Now
 suppose that $0<s<t_1+\cdots +t_r$. By the definition of $c_s$ there exist $i_1,\cdots,i_r$ such that $s=i_1+\cdots + i_r$ and $c_s=b_{i_1}+\cdots +b_{i_r}$,
 where $0<i_j<t_j$ for some $j=1,\cdots, r$. So $s+1=i_1+\cdots +(i_j+1)+\cdots + i_r$. Therefore
$$c_s=b_{i_1}+\cdots + b_{i_j}+\cdots +b_{i_r}>b_{i_1}+\cdots + b_{{i_j}+1}+\cdots +b_{i_r}=c_{s+1}.$$
Hence $c_s>c_{s+1}$, as desired.
\end{proof}

\begin{Corollary}\label{brexit}
With the notation and the assumptions of the Theorem~\ref{back}, we have
$$\mu(I_1\ldots I_r)=\sum_{j=1}^r\mu(I_j)-(r-1).$$
In particular, if $I$ is a monomial ideal with $G(I)=\{x^{ia}y^{b_i}\}_{i=s,\ldots,t}$, then $\mu(I^k)=k(t-s)+1$.
\end{Corollary}

As a consequence of Lemma~\ref{Theresa May} and Corollary~\ref{brexit}, we obtain

\begin{Corollary}\label{Crash}
If $I$ is a lexsegment ideal, then $\mu(I^2)=2\mu(I)-1$.
\end{Corollary}

We would like to remark that this result  also follows from Proposition~\ref{huneke}, since lexsegment ideals are contracted, see \cite[Proposition 3.3]{CNJR}.

\medskip
Corollary~\ref{Crash} cannot be extended to lexsegment ideals in more than two variables. Indeed, the lexsegment ideal given  in the example after Corollary~\ref{true} shows this.

\begin{Remark}\label{generators}{\em
Let $I\subset K[x,y]$ be a non-principal monomial ideal,  and let $G(I)=\{u_i\:\; i=1,\ldots, m\}$ be as before. The safe area argument shows that if $m\geq 3$,
then no $u_iu_j$ different from the monomials $u_1^2,u_1u_2,u_{m-1}u_m$ and  $u_m^2$ divides any one of them, and they do not divide each other. Therefore,
$\{u_1^2,u_1u_2,u_{m-1}u_m, u_m^2\} \subseteq G(I^2)$. In particular, $\mu(I^2)\geq 4$ for $m\geq 3$.}
\end{Remark}
\medskip
Let $S$ be a nonempty subset of $T(I)$. We set $\bar{S}=\bigcap_{u_iu_j\in S}S_{ij}$.

\begin{Lemma}\label{bar}
Let $\emptyset\neq S\subseteq T(I)$. If $\bar {S}\neq \emptyset$, then $S$ is not a set of generators of $I^2$.
\end{Lemma}

\begin{proof}
Since $\bar{S}\neq \emptyset$, there exist $k\leq \ell$ such that $u_ku_{\ell}\in \bar{S}$. By the definition of the safe area it follows that no monomial in $S$
can divide $u_ku_{\ell}$. Therefore $S$ can not be a set of generator of $I^2$.
\end{proof}

Even though we don't have  $\mu(I^2)>\mu(I)$ for non-principal ideals, we can show

\begin{Proposition}
\label{small}
Let $I\subset K[x,y]$ be a non-principal monomial ideal with  $\mu(I)\leq 7$. Then $\mu(I^2)>\mu(I)$.
\end{Proposition}

\begin{proof}
As before we may assume that $G(I)=\{u_i\:\; i=1,\ldots, m\}$ where $u_i=x^{a_i}y^{b_i}$, $a_1>a_2>\ldots >a_{m-1}>a_m=0$ and $0=b_1<b_2<\ldots<b_m$.
If $\mu(I)=2$, then $\mu(I^2)=3$, and if $\mu(I)=3$ the assertion holds by the Remark~\ref{generators}. Now let $4 \leq \mu(I)\leq 6$, and choose $S\subset
T(I)$
with $|S|=\mu(I)$. The reader can check that $\bar{S}\neq \emptyset$. Therefore, Lemma~\ref{bar} implies that $S$ is not a system of generators of $I^2$. Hence
$\mu(I^2)>\mu(I)$ for $\mu(I)=4,5$ or $6$.

Finally,  let $\mu(I)=7$. Then the only set $S$ with $|S|=7$ and $\bar{S}=\emptyset$ is the set $$S=\{u_1^2,u_1u_2,u_2^2,u_6^2,u_1u_7,u_{6}u_7,u_7^2\}.$$ But
again this can not be a set of generators of $I^2$ because otherwise, we have the following relations,
$$2a_2\leq a_1+a_6,\quad 2a_6\leq a_5\quad and \quad a_1\leq a_5+a_6,$$
which implies that $a_2\leq a_5$, a contradiction.
\end{proof}

\section{On the Cohen-Macaulay type of a product of monomial ideals}

A result of Huneke \cite{Hu} says that if $I$ is a graded Gorenstein ideal in the polynomial ring $S=K[x_1,\ldots,x_n]$ with $\height(I)\geq 2$, then $I$ is not
a product of two proper ideals. This theorem implies that if $n\geq 2$ and $\height I\geq 2$, then the Cohen-Macaulay type of $S/I^2$ is greater than or equal to
$2$. Recall that the Cohen-Macaulay type $\type(M)$ of a graded Cohen-Macaulay $S$-module $M$, is by definition the $K$-vector space dimension of
$\Ext_{S}^d(K,M)$. It is known that  if $$0\to F_p\to \cdots \to F_0\to M\to 0$$ is the minimal graded free $S$-resolution of $M$, then $\type(M)=\rank(F_p)$.

In the case that $M=S/I$ with $\height(I)=n$,  $\type(S/I)$ is the smallest number $r$ such that $I=Q_1\cap Q_2\cap \ldots\cap Q_r$, where each $Q_i$ is an
irreducible ideal. In the case that $I$ is a monomial ideal, irreducible ideals in this intersection are all of the form
$(x_1^{a_1},x_2^{a_2},\ldots,x_n^{a_n})$.

\medskip
In the following very special case we can improve Huneke's lower  bound for the Cohen--Macaulay type as follows.

\begin{Theorem}\label{type}
Let $K$ be a field, $S=K[x_1,\ldots,x_n]$ the polynomial ring with $n\geq 2$ variables and let $I, J$ be monomials ideal of height $n$. Then $\type (S/IJ)\geq 3$
if either $n\geq 3$, or else $n=2$ and $\max\{\mu(I),\mu(J)\}\geq 3$.
\end{Theorem}

Before proving this theorem, we give a rough lower bound of the number of generators of $IJ$. We first show

\begin{Lemma}\label{baby}
Let $I,J\subset S=K[x,y]$ be monomial ideals of height $2$. Then $\mu(IJ)\geq 3$. Moreover, if $\max\{\mu(I),\mu(J)\}\geq 3$, then $\mu(IJ)\geq 4$.
\end{Lemma}

\begin{proof}
Let $G(I)=\{x^{a_1},y^{a_2},\ldots\}$ and $G(J)=\{x^{b_1},y^{b_2},\ldots\}$. Then,
 $$IJ=(x^{a_1+b_1},y^{a_2+b_2},x^{a_1}y^{b_2},x^{b_1}y^{a_2},\dots).$$
Now suppose that $\mu(IJ)=2$. So $G(IJ)=\{x^{a_1+b_1},y^{a_2+b_2}\}$, which is impossible because $x^{a_1}y^{b_2}\notin (x^{a_1+b_1},y^{a_2+b_2})$.

Next suppose that $\max\{\mu(I),\mu(J)\}\geq 3$. We may assume that $\mu(I)\geq 3$. Further we may assume that $G(I)=\{u_1,\ldots,u_m\}$, $m\geq 3$, with
$u_1=x^{a_1}, u_m=y^{a_2}$ and $G(J)=\{v_1,\ldots,v_n\}$, $n\geq 2$, with $v_1=x^{b_1}, v_n=y^{b_2}$.

Now assume that $|G(IJ)|<4$. So $|G(IJ)|=3$ by the first part of the proof. Then, $G(IJ)=\{u_1v_1,u_mv_n, u_iv_j\}$, where $u_iv_j\neq u_1v_1,u_mv_n$. Since
$m\geq 3$, there exists $k\neq j$ such that $u_iv_k\neq u_1v_1, u_mv_n$. This is clear that $u_iv_j\nmid u_iv_k$. Also $u_1v_1, u_mv_n$ do not divide $u_iv_k$,
because $u_iv_k$ is not a pure power. This shows that $G(IJ)\neq\{u_1v_1,u_mv_n, u_iv_j\}$, contradiction.
\end{proof}

\begin{Proposition}\label{rough}
Let $I,J\subset S$ be as in Theorem~\ref{type}. Then, $\mu(IJ)\geq n+{n\choose 2}$.
\end{Proposition}

\begin{proof}
Let $i$ and $j$ be integers with the property $1\leq i<j\leq n$, and let $\phi_{ij}: S\to K[x_i,x_j]$ be the $K$-algebra homomorphism with $\phi_{ij}(x_k)=0$ if
$k\neq i,j$ and $\phi_{ij}(x_i)=x_i$ and $\phi_{ij}(x_j)=x_j$. Let $G(I)=\{x_1^{a_1},x_2^{a_2}\ldots, x_n^{a_n}, \ldots\}$ and $G(J)=\{x_1^{b_1},x_2^{b_2}\ldots,
x_n^{b_n}, \ldots\}$. Then $\phi_{ij}(I)=(x_i^{a_i},x_j^{a_j},\ldots)$ and $\phi_{ij}(J)=(x_i^{b_i},x_j^{b_j},\ldots)$. Now we apply Lemma~\ref{baby} and obtain
$$\phi_{ij}(IJ)=\phi_{ij}(I)\phi_{ij}(J)=(x_i^{a_i+b_i},x_j^{a_j+b_j},u_{ij},\ldots),$$
where $u_{ij}\in K[x_i,x_j]$.

Since $G(\phi_{ij}(IJ))\subset G(IJ)$, it follows that
 $$G(IJ)=\{x_1^{a_1+b_1},\ldots , x_n^{a_n+b_n}, \{u_{ij}\}_{1\leq i<j\leq n},\ldots\}.$$
This yields the desired inequality.
\end{proof}

\begin{proof}[Proof of Theorem~\ref{type}]
Suppose $\type(S/IJ)<3$. If $\type(S/IJ)=1$, then $I$ is an irreducible ideal and hence is generated by $n$ elements. This contradicts Proposition~\ref{rough}. Let
us assume that $\type(S/IJ)=2$. Then $IJ=Q_1\cap Q_2$ with $Q_1=(x_1^{c_1},\ldots,x_n^{c_n})$ and $Q_2=(x_1^{d_1},\ldots,x_n^{d_n})$. We may assume that $c_i>d_i$
for $1\leq i\leq r$, $c_i<d_i$ for $r+1\leq i\leq s$ and $c_i=d_i$ for $s+1\leq i\leq n$. It follows that
$$G(IJ)=\{x_1^{c_1},\ldots,x_r^{c_r},x_{r+1}^{d_{r+1}},\ldots,x_n^{d_n}\}\cup\{x_i^{d_i}y_j^{c_j}\;\: i=1,\ldots,r \text{ and } j=r+1,\ldots, s\}.$$

Therefore, $\mu(IJ)=n+r(s-r)$. On the other hand, by Proposition~\ref{rough}, $\mu(IJ)\geq n+{n\choose 2}$. This implies that $r(s-r)\geq {n\choose 2}$. For
given value of $s$ the maximal value of $r(s-r)$ is $s^2/4$, and since $s\leq n$, we have $r(s-r)\leq n^2/4$. Since $n^2/4<{n\choose 2}$ if $n\geq 3$, we get a
contradiction.

In the case that $n=2$ we assume that  $\max\{\mu(I),\mu(J)\}\geq 3$. Then Lemma~\ref{baby}   implies that $\mu(IJ)\geq 4$. Let $m=\mu(IJ)$. Since $n=2$, the
minimal free resolution of $S/IJ$ is given by
$$0\to S^{m-1}\to S^m \to S \to S/IJ \to 0.$$
This implies that $\type(S/IJ)=m-1\geq 3$.
\end{proof}

\end{document}